\documentclass[10pt,letterpaper]{article}
\usepackage[utf8]{inputenc}
\usepackage{amsmath}
\usepackage{amsfonts}
\usepackage{amssymb}
\usepackage{graphicx}
\usepackage{multicol}
\usepackage{mathrsfs}
\usepackage{upref,amsthm,amsxtra,exscale}
\usepackage{cite}
\usepackage[colorlinks=true,urlcolor=blue,
citecolor=red,linkcolor=blue,linktocpage,pdfpagelabels,
bookmarksnumbered,bookmarksopen]{hyperref}

\newtheorem{theorem}{Theorem}[section]
\newtheorem{corollary}[theorem]{Corollary}

\newtheorem{lemma}[theorem]{Lemma}
\newtheorem{proposition}[theorem]{Proposition}

\numberwithin{equation}{section}

\author{Omar Cabrera \quad and \quad Mónica Clapp\footnote{M. Clapp was partially supported by UNAM-DGAPA-PAPIIT grant IN100718 (Mexico).}}
\title{Multiple solutions to weakly coupled supercritical elliptic systems}
\date{\today}

\begin{document}
\maketitle

\begin{abstract}
We study a weakly coupled supercritical elliptic system of the form
\begin{equation*}
\begin{cases}
-\Delta u = |x_2|^\gamma \left(\mu_{1}|u|^{p-2}u+\lambda\alpha |u|^{\alpha-2}|v|^{\beta}u \right) & \text{in }\Omega,\\
-\Delta v = |x_2|^\gamma \left(\mu_{2}|v|^{p-2}v+\lambda\beta |u|^{\alpha}|v|^{\beta-2}v \right) & \text{in }\Omega,\\
u=v=0 & \text{on }\partial\Omega,
\end{cases}
\end{equation*}
where $\Omega$ is a bounded smooth domain in $\mathbb{R}^{N}$, $N\geq 3$, $\gamma\geq 0$, $\mu_{1},\mu_{2}>0$, $\lambda\in\mathbb{R}$, $\alpha, \beta>1$, $\alpha+\beta = p$, and $p\geq 2^{*}:=\frac{2N}{N-2}$.

We assume that $\Omega$ is invariant under the action of a group $G$ of linear isometries, $\mathbb{R}^{N}$ is the sum $F\oplus F^\perp$ of $G$-invariant linear subspaces, and  $x_2$ is the projection onto $F^\perp$ of the point $x\in\Omega$. 

Then, under some assumptions on $\Omega$ and $F$, we establish the existence of infinitely many fully nontrivial $G$-invariant solutions to this system for $p\geq 2^*$ up to some value which depends on the symmetries and on $\gamma$. Our results apply, in particular, to the system with pure power nonlinearity ($\gamma=0$), and yield new existence and multiplicity results for the supercritical Hénon-type equation
$$-\Delta w = |x_2|^\gamma \,|w|^{p-2}w \quad\text{in }\Omega, \qquad w=0 \quad\text{on }\partial\Omega.$$


\end{abstract}

\section{Introduction}

We consider the weakly coupled elliptic system
\begin{equation} \label{eq:system}
\begin{cases}
-\Delta u = |x_2|^\gamma \left(\mu_{1}|u|^{p-2}u+\lambda\alpha |u|^{\alpha-2}|v|^{\beta}u \right) & \text{in }\Omega,\\
-\Delta v = |x_2|^\gamma \left(\mu_{2}|v|^{p-2}v+\lambda\beta |u|^{\alpha}|v|^{\beta-2}v \right) & \text{in }\Omega,\\
u=v=0 & \text{on }\partial\Omega,
\end{cases}
\end{equation}
where $\Omega$ is a bounded smooth domain in $\mathbb{R}^{N}$, $N\geq 3$, $\gamma\geq 0$, $\mu_{1},\mu_{2}>0$, $\lambda\in\mathbb{R}$, $\alpha, \beta>1$, $\alpha+\beta = p$, and $p\in (2,\infty)$. The space $\mathbb{R}^{N}$ is decomposed into a direct sum $\mathbb{R}^{N} = F \oplus F^\perp$, where $F^\perp$ is the orthogonal complement of $F$, and $x_2$ is the orthogonal projection onto $F^\perp$ of the point $x\in\Omega$.

Systems of this type arise as a model for various physical phenomena. In particular, the cubic system, where $N=3$, $p=4$, $\alpha=\beta$ and $\gamma=0$, appears in nonlinear optics and in the study of standing waves in a double mixture of Bose-Einstein condensates, and has received much attention in recent years. There is an extensive literature on subcritical systems with $p<2_N^*:=\frac{2N}{N-2}$ and $\gamma=0$. We refer to \cite{so} for a detailed account. 

When $\gamma=0$ and $p$ is the critical Sobolev exponent $2_N^*$, existence and multiplicity results, both in bounded domains and in $\mathbb{R}^N$, were recently obtained in \cite{cf,cp,glw,ppw,ps}. Critical systems of Brezis-Nirenberg type have been studied in \cite{cz1,cz2,llw,pt}.

Here we shall, mainly, focus our attention on the supercritical case $p>2_N^*$.

When $\lambda=0$ the system \eqref{eq:system} reduces to the problem 
\begin{equation} \label{eq:Henon}
\begin{cases}
-\Delta w = |x_2|^\gamma \,|w|^{p-2}w &\text{in }\Omega, \\
w=0 &\text{on }\partial\Omega.
\end{cases}
\end{equation}
Note that, if $w$ solves \eqref{eq:Henon}, then $(\mu_{1}^{1/(2-p)}w,0)$ and $(0,\mu_{2}^{1/(2-p)}w)$ solve the system \eqref{eq:system} for every $\lambda$. Solutions of this type are called \emph{semitrivial}. We are interested in \emph{fully nontrivial} solutions to \eqref{eq:system}, i.e., solutions where both components, $u$ and $v$, are nontrivial. A solution is said to be \emph{synchronized} if it is of the form $(sw,tw)$ with $s,t\in\mathbb{R}$, and it is called \emph{positive} if $u\geq0$ and $v\geq0$. The system \eqref{eq:system} is called \emph{cooperative} if $\lambda>0$ and \emph{competitive} if $\lambda<0$.

In the cooperative case, we make the following additional assumption:

\begin{itemize}
\item[$(A)$] If $\lambda>0$, then there exists $r\in(0,\infty)$ such that
$$\mu_1 r^{p-2} + \lambda\alpha r^{\alpha -2}-\lambda\beta r^\alpha -\mu_2=0.$$
\end{itemize}

We consider symmetric domains. Our setting is as follows.

Let $G$ be a closed subgroup of the group $O(N)$ of linear isometries of $\mathbb{R}^N$. We write $Gx:=\{gx:g\in G\}$ for the $G$-orbit of a point $x\in\mathbb{R}^{N}$. Recall that a subset $X$ of $\mathbb{R}^N$ is called $G$-invariant if $Gx\subset X$ for every $x\in X$ and a function $u:X\to \mathbb{R}$ is $G$-invariant if $u$ is constant on $Gx$ for every $x\in X$.

We assume that domain $\Omega$ and the linear subspace $F$ of $\mathbb{R}^N$ are $G$-invariant, and satisfy
\begin{itemize}
\item[$(F_1)$]$F\neq\mathbb{R}^N$ if $\gamma>0$,
\item[$(F_2)$]$\Omega_0:=\{x\in\Omega:\dim Gx=0\}\subset F$.
\end{itemize}
We are interested in finding $G$-invariant solutions $(u,v)$ to the system \eqref{eq:system}, i.e., both components $u$ and $v$ are $G$-invariant. We denote by
$$d:= \min \{\dim Gx:x\in\Omega\smallsetminus\Omega_0\}>0,$$
and, for $p\in [1,\infty)$, we set
$$\gamma_p:= p\left(\frac{N}{2}-\frac{N}{p}-1\right).$$
We write $2_k^*$ for the critical Sobolev exponent in dimension $k$, i.e., $2_k^*:=\frac{2k}{k-2}$ if $k>2$ and $2_k^*:=\infty$ if $k\leq 2$. Note that $\gamma_p\geq0$ if $p\geq 2^*_N$. We will prove the following result.

\begin{theorem} \label{thm:mainthm}
Assume $(A)$, $(F_1)$ and $(F_2)$, and let $p\in (2,2^*_{N-d})$. If $\Omega_0\neq\emptyset$ we assume further that $\gamma>\max\{\gamma_p,0\}$. Then, the system \eqref{eq:system} has infinitely many fully nontrivial $G$-invariant solutions, one of which is positive.
\end{theorem}

Note that, as $d>0$, we have that $2^*_{N-d}>2^*_N$. For $\lambda>0$ the solutions given by Theorem \ref{thm:mainthm} are syncronized and infinitely many of them are sign-changing. In contrast, as shown in \cite[Proposition 2.3]{cp}, there are no syncronized solutions for $\lambda$ smaller than some number $\lambda_*<0$. Moreover, the positive solution given by Theorem \ref{thm:mainthm} has minimal energy among all fully nontrivial $G$-invariant solutions when $\lambda<0$.

We state some special cases of Theorem \ref{thm:mainthm}. Firstly, if $\Omega_0=\emptyset$ we may take $\gamma=0$, and our result reads as follows.

\begin{corollary} \label{cor:pure_system}
Assume $(A)$. If $\dim Gx\geq d>0$ for every $x\in\Omega$, then, for any $p\in [2^*_N,2^*_{N-d})$, the system
\begin{equation} \label{eq:pure_system}
\begin{cases}
-\Delta u = \mu_{1}|u|^{p-2}u+\lambda\alpha |u|^{\alpha-2}|v|^{\beta}u & \text{in }\Omega,\\
-\Delta v = \mu_{2}|v|^{p-2}v+\lambda\beta |u|^{\alpha}|v|^{\beta-2}v  & \text{in }\Omega,\\
u=v=0 & \text{on }\partial\Omega,
\end{cases}
\end{equation}
has infinitely many fully nontrivial $G$-invariant solutions, one of which is positive.
\end{corollary}

For $p\in(2,2_N^*)$ this result is true without any symmetry assumption. For $p=2^*_N$ it was proved in \cite[Corollary 1.3]{cf}. 

Taking $\mu_1=\mu_2=1$ and $\lambda=0$ in Theorem \ref{thm:mainthm}, we get a multiplicity result for problem \eqref{eq:Henon}. We shall prove, in fact, the following improvement of it, that states the existence of infinitely many sign-changing solutions. 

\begin{theorem} \label{thm:equation}
Assume $(F_1)$ and $(F_2)$, and let $p\in (2,2^*_{N-d})$. If $\Omega_0\neq\emptyset$ we also assume that $\gamma>\max\{\gamma_p,0\}$. Then, the problem \eqref{eq:Henon} has a positive $G$-invariant solution which has least energy among all nontrivial $G$-invariant solutions, and infinitely many sign-changing $G$-invariant solutions.
\end{theorem} 

If $\Omega_0=\emptyset$ we may take $\gamma=0$, and \eqref{eq:Henon} becomes
\begin{equation*} 
-\Delta w = |w|^{p-2}w \;\text{in }\Omega, \qquad w=0 \; \text{on }\partial\Omega.
\end{equation*}
For this special case, Theorem \ref{thm:equation} was proved in \cite[Theorem 2.3]{cpa}. The method that we will use to prove our results is an extension of the method used in \cite{cpa}.

If $\gamma>0$ and $\Omega_0\neq\emptyset$, then $\Omega\cap F\neq\emptyset$ and problem \eqref{eq:Henon} is of Hénon-type. When $\Omega$ is the unit ball $B$ and $F=\{0\}$ it is the well known Hénon problem
\begin{equation} \label{eq:Bhenon}
-\Delta w = |x|^\gamma|w|^{p-2}w \;\text{in }B, \qquad w=0 \; \text{on }\partial B,
\end{equation}
which has been widely studied, starting with the pioneering work \cite{ni} of W.-M. Ni, who proved the existence of a positive radial solution if $\gamma>\gamma_p$ or, equivalently, if $p<p_\gamma:=\frac{2(N+\gamma)}{N-2}$. A Pohozhaev-type identity shows that \eqref{eq:Bhenon} does not have a nontrivial solution if $p\geq p_\gamma$; cf. Proposition \ref{prop:pohozhaev} below.

Other special cases of Theorem \ref{thm:equation} are given in \cite{bs,dp}. In \cite{bs} Badiale and Serra established the existence of a positive $G$-invariant solution to \eqref{eq:Bhenon} for the group $G=O(m)\times O(n)$, $m+n=N$, and $p$ and $\gamma$ as in Theorem \ref{thm:equation}. In \cite{dp} dos Santos and Pacella studied problem \eqref{eq:Bhenon} when $N=2m$, $G=O(m)\times O(m)$, and $F$ is either $\{0\}$ or $\mathbb{R}^m\times \{0\}$. For $p\in(2,2_{m+1}^*)$ and large enough $\gamma>0$, they established the existence of a positive least energy $G$-invariant solution which blows up at a $G$-orbit of minimal dimension in $\partial B\smallsetminus F$ as $\gamma\to\infty$. We believe that a similar blow-up behavior is also true in our more general setting. 

The Hénon problem \eqref{eq:Bhenon} has also been studied in general bounded domains without any symmetries, and bubbling solutions have been constructed for exponents $p$ which are, either close to $2^*_N$, or slightly below the critical Hénon exponent $p_\gamma$. We refer to the recent papers \cite{clp, dfm, ggn} for a detailed account.

As was shown in \cite{ctv} by Conti, Terracini and Verzini for a subcritical system, the positive least energy solutions to the supercritical system \eqref{eq:system} exhibit also phase separation as $\lambda\to -\infty$. More precisely, one has the following result.

\begin{theorem} \label{thm:separation}
Assume that, for some sequence $(\lambda_{k})$ with $\lambda_{k}\to -\infty$, there exists a positive fully nontrivial $G$-invariant solution $(u_{k},v_{k})$ to the system \eqref{eq:system} with $\lambda=\lambda_{k}$, which has least energy among all fully nontrivial $G$-invariant solutions to that system. Then, after passing to a subsequence, we have that
\begin{itemize}
\item[$(a)$]$u_{k}\to u_{\infty}$ and $v_{k}\to v_{\infty}$ strongly in $D_0^{1,2}(\Omega)$,
\item[$(b)$]$u_{\infty}$ and $v_{\infty}$ are $G$-invariant, $u_{\infty}\geq 0$, $v_{\infty}\geq 0$ and $u_{\infty}v_{\infty}\equiv 0$,
\item[$(c)$]$u_{\infty}-v_{\infty}$ is a least energy $G$-invariant sign-changing solution to the problem
\begin{equation} \label{eq:signchanging}
\begin{cases}
-\Delta w = |x_2|^\gamma\left(\mu_{1}|w^+|^{p-2}w^+ + \mu_{2}|w^-|^{p-2}w^-\right) &\text{ in } \Omega, \\
w=0 &\text{ on } \partial\Omega,
\end{cases}
\end{equation}
where $w^{+} := \max \{w,0\}$ and $w^{-} := \min \{w,0\}$.
\end{itemize}
\end{theorem}

This paper is organized as follows. In Section \ref{sec:compactness} we show that our assumptions on the symmetries yield a good variational setting for the system \eqref{eq:system} and we discuss the variational problem. Section \ref{sec:competitive} is devoted to the proofs of Theorem \ref{thm:mainthm} for $\lambda<0$ and Theorem \ref{thm:separation}. In Section \ref{sec:cooperative} we prove Theorem \ref{thm:equation} and we derive Theorem \ref{thm:mainthm} for $\lambda>0$ from it.

\section{The symmetric variational setting} \label{sec:compactness}

Let $G$ be a closed subgroup of $O(N)$, $F$ be a $G$-invariant linear subspace of $\mathbb{R}^N$ and $\Omega$ be a $G$-invariant bounded smooth domain in $\mathbb{R}^N$, which satisfy $(F_1)$ and $(F_2)$.

For $1\leq p <\infty$ and $\gamma\geq 0$ we denote by 
$$L^p(\Omega;|x_2|^\gamma):=\{w:\Omega\to\mathbb{R}:|x_2|^{\gamma/p}w\in L^p(\Omega)\}$$
the weighted Lebesque space with the norm given by 
$$|w|_{p,\gamma}:=\left(\int_\Omega |x_2|^{\gamma}|w|^p\right)^{1/p}.$$
Note that assumption $(F_1)$ guarantees that this is, indeed, a norm. As usual, we write $D^{1,2}_0(\Omega)$ for the closure of $\mathcal{C}_c^\infty(\Omega)$ in the Sobolev space $D^{1,2}(\mathbb{R}^N):=\{w\in L^{2^*}:\nabla w\in L^2(\mathbb{R}^N,\mathbb{R}^N)\}$, equiped with the norm 
$$\|w\|:=\left(\int_{\mathbb{R}^N}|\nabla w|^2\right)^{1/2}.$$

A \textit{(weak) solution to the system} \eqref{eq:system} is a pair $(u,v)$ such that $u,v\in D^{1,2}_0(\Omega)\cap L^p(\Omega;|x_2|^\gamma)$, and they satisfy the identities
\begin{align}
&\int_{\Omega}\nabla u\cdot\nabla\vartheta - \mu_1\int_{\Omega}|x_2|^{\gamma}|u|^{p-2}u\vartheta - \lambda \alpha \int_{\Omega} |x_2|^{\gamma}|u|^{\alpha-2}u|v|^{\beta}\vartheta = 0, \label{eq:sol1}\\
&\int_{\Omega}\nabla v\cdot\nabla\vartheta - \mu_2\int_{\Omega}|x_2|^{\gamma}|v|^{p-2}v\vartheta - \lambda \beta \int_{\Omega} |x_2|^{\gamma}|u|^{\alpha}|v|^{\beta-2}v\vartheta=0, \label{eq:sol2}
\end{align}
for every $\vartheta\in \mathcal{C}_c^\infty(\Omega)$.

Hebey and Vaugon showed in \cite{hv} that the Sobolev embedding and the Rellich-Kondrachov theorems can be improved in $G$-invariant domains whose $G$-orbits have positive dimension. Ivanov and Nazarov obtained an extension of this result in \cite{in}, which allows to consider domains with finite $G$-orbits. These results will play a crucial role in the proof of our main result. The version that we need will be derived from them next.

Set
$$D^{1,2}_0(\Omega)^G:=\{w\in D^{1,2}_0(\Omega):w \text{ is } G\text{-invariant}\},$$
and recall the definitions of\; $\Omega_0:=\{x\in\Omega:\dim Gx=0\}$,
$$d:= \min \{\dim Gx:x\in\Omega\smallsetminus\Omega_0\}, \qquad \gamma_p:= p\left(\frac{N}{2}-\frac{N}{p}-1\right),$$
$2_k^*:=\frac{2k}{k-2}$ if $k>2$ \, and \, $2_k^*:=\infty$ if $k\leq 2$.

\begin{theorem}[Hebey-Vaugon 1997; Ivanov-Nazarov 2016] \label{thm:in}
If $p\in \mathbb{R}\cap[1,2_{N-d}^*]$ and $\gamma\geq \max\{\gamma_p,0\}$, then the embedding
$$D^{1,2}_0(\Omega)^G \hookrightarrow L^p(\Omega;|x_2|^\gamma)$$
is continuous. Moreover, this embedding is compact if $p\in[1,2_{N-d}^*)$ and, either $\Omega_0=\emptyset$, or $\Omega_0\neq\emptyset$ and $\gamma > \max\{\gamma_p,0\}$.
\end{theorem}

\begin{proof}
If $\Omega_0=\emptyset$ and $\gamma=0$ this statement is proved in \cite[Corollary 2]{hv}. The result for $\Omega_0=\emptyset$ and $\gamma>0$ follows immediately from it.

If $\Omega_0\neq\emptyset$ and $\gamma > \max\{\gamma_p,0\}$, let $\mathfrak{r}(x)$ denote the Riemannian distance in $\Omega$ from $x$ to $\Omega_0$. The statement of this theorem with $L^p(\Omega;|x_2|^\gamma)$ replaced by $L^p(\Omega;\mathfrak{r}(x)^\gamma)$ was proved in \cite[Theorem 2.4]{in}. Since assumption $(F_2)$ implies that $|x_2|\leq \mathfrak{r}(x)$ for every $x\in\Omega$, our claim follows.
\end{proof}

From now on we will assume that $p \in (2, 2^*_{N-d})$ and that $\gamma > \max\{\gamma_p,0\}$ if $\Omega_0\neq\emptyset$. We write
$$S_{p,\gamma}^G:=\inf_{\substack{u \in D_0^{1,2}(\Omega)^G \\ u\neq 0}}\, \frac{\|u\|^2}{\;|u|_{p,\gamma}^2}$$
for the best constant for the embedding $D_0^{1,2}(\Omega)^G \hookrightarrow L^p(\Omega; |x_2|^\gamma)$.

Let $\mathscr{D}^G:=D_0^{1,2}(\Omega)^G\times D_0^{1,2}(\Omega)^G$. Theorem \ref{thm:in} guarantees that the functional $E : \mathscr{D}^G\to \mathbb{R}$, given by
\begin{align*}
E(u, v) &:= \frac{1}{2} \int_{\Omega} \left( |\nabla u|^2 + |\nabla v|^2 \right) - \frac{1}{p}\int_{\Omega} |x_2|^\gamma\left( \mu_1 |u|^p + \mu_2 |v|^p \right)\\
& \qquad - \lambda \int_{\Omega}|x_2|^\gamma |u|^{\alpha}|v|^{\beta}
\end{align*}
is well defined. It is of class $\mathcal{C}^1$ and has the following property. 

\begin{proposition} \label{prop:solutions}
The critical points of $E:\mathscr{D}^G\to \mathbb{R}$ are the $G$-invariant solutions to the system \eqref{eq:system}.
\end{proposition}

\begin{proof}
Let $(u,v)$ be a critical point of $E:\mathscr{D}^G\to \mathbb{R}$. Then, 
\begin{align*}
\partial_u E(u,v)\varphi=&\int_{\Omega}\nabla u\cdot\nabla\varphi - \mu_1\int_{\Omega}|x_2|^{\gamma}|u|^{p-2}u\varphi - \lambda \alpha \int_{\Omega} |x_2|^{\gamma}|u|^{\alpha-2}u|v|^{\beta}\varphi = 0,\\
\partial_v E(u,v)\varphi=&\int_{\Omega}\nabla v\cdot\nabla\varphi - \mu_2\int_{\Omega}|x_2|^{\gamma}|v|^{p-2}v\varphi - \lambda \beta \int_{\Omega} |x_2|^{\gamma}|u|^{\alpha}|v|^{\beta-2}v\varphi=0, 
\end{align*}
for every $G$-invariant function $\varphi\in \mathcal{C}_c^\infty(\Omega)$.

By Theorem \ref{thm:in} we have that $u,v\in L^p(\Omega;|x_2|^\gamma)$. So we need only to prove that the identities \eqref{eq:sol1} and \eqref{eq:sol2} hold true for every $\vartheta\in \mathcal{C}_c^\infty(\Omega)$.

Let $\vartheta\in \mathcal{C}_c^\infty(\Omega)$, and define
$$\varphi(x):=\frac{1}{\mu(G)}\int_G\vartheta(gx)\mathrm{d}\mu,$$
where $\mu$ is the Haar measure on $G$; see \cite{n}. Then, $\varphi$ is $G$-invariant. A straightforward computation yields
\begin{align*}
0=\partial_u E(u,v)\varphi=&\int_{\Omega}\nabla u\cdot\nabla\vartheta - \mu_1\int_{\Omega}|x_2|^{\gamma}|u|^{p-2}u\vartheta - \lambda \alpha \int_{\Omega} |x_2|^{\gamma}|u|^{\alpha-2}u|v|^{\beta}\vartheta,\\
0=\partial_v E(u,v)\varphi=&\int_{\Omega}\nabla v\cdot\nabla\vartheta - \mu_2\int_{\Omega}|x_2|^{\gamma}|v|^{p-2}v\vartheta - \lambda \beta \int_{\Omega} |x_2|^{\gamma}|u|^{\alpha}|v|^{\beta-2}v\vartheta,
\end{align*}
cf. \cite[Lemma 2.1]{cr}. This completes the proof.
\end{proof}

We define
\begin{align*}
f(u, v) &:= \partial_u E(u, v) u = \int_{\Omega} |\nabla u|^2 - \mu_1 \int_{\Omega}|x_2|^\gamma|u|^{p} - \lambda \alpha \int_{\Omega}|x_2|^\gamma|u|^{\alpha}|v|^{\beta}, \\
h(u, v) &:= \partial_v E(u, v) v = \int_{\Omega} |\nabla v|^2 - \mu_2 \int_{\Omega}|x_2|^\gamma|v|^{p} - \lambda \beta \int_{\Omega} |x_2|^\gamma|u|^{\alpha}|v|^{\beta},
\end{align*}
and consider the set
$$\mathcal{N}^G := \{ (u, v) \in \mathscr{D}^G: u \neq 0,\, v \neq 0,\, f(u, v) = h(u, v) = 0 \}.$$
The following proposition improves \cite[Proposition 2.1]{cp}, as it allows $\alpha,\beta>2$.

\begin{proposition}{\label{prop:nehari}}
If $\lambda<0$, the following statements hold true:
\begin{itemize}
\item[$(a)$]There exists $c_0>0$ such that, for every $(u,v) \in \mathcal{N}^G$,
$$c_0 \leq \|u\|^2 \leq \mu_1|u|_{p,\gamma}^p \quad \text{and} \quad c_0 \leq \|v\|^2\leq \mu_2|v|_{p,\gamma}^p.$$
\item[$(b)$]$\mathcal{N}^G$ is a closed $\mathcal{C}^1$-submanifold of codimension $2$ of $\mathscr{D}^G$. More precisely, $\nabla f(u,v)$ and $\nabla h(u,v)$ are linearly independent and generate the orthogonal complement of the tangent space to $\mathcal{N}^G$ at $(u,v)$.
\item[$(c)$]$\mathcal{N}^G$ is a natural constraint for $E: \mathscr{D}^G\to \mathbb{R}$, i.e., the critical points of the restriction of $E$ to $\mathcal{N}^G$ are critical points of $E$.
\end{itemize}
\end{proposition}
\begin{proof}
$(a):$ Let $(u,v) \in \mathcal{N}^G$. Then, as $\lambda < 0$, we have that $\|u\|^2 \leq \mu_1|u|_{p,\gamma}^p$. Therefore,
$$0 < S_{p,\gamma}^G  \leq \frac{\|u\|^2}{\;|u|_{p,\gamma}^2} \leq \mu_1^{2/p} \left(\|u\|^2\right)^{(p-2)/p}.$$
Multiplying this inequality by $\mu_1^{-2/p}$ and raising it to the power of $p/(p-2)$ we obtain the statement for $u$. An analogous argument can be used for $v$.

$(b):$ First, notice that $\mathcal{N}^G\neq \emptyset$. Indeed, if $\varphi_1, \varphi_2 \in \mathcal{C}_c^{\infty}(\Omega)$ are nontrivial $G$-invariant functions with disjoint supports and $s_1,s_2>0$ are such that $\|s_i\varphi_i\|^2=\mu_i|s_i\varphi_i|^p_{p,\gamma}$, then $(s_1 \varphi_1, s_2\varphi_2) \in \mathcal{N}^G$.

Statement $(a)$ implies that $\mathcal{N}^G$ is closed in $\mathscr{D}^G$. Next, we prove that $\nabla f(u,v)$ and $\nabla h(u,v)$ are linearly independent if $(u,v)\in\mathcal{N}^G$. Assume there exist $s, t \in \mathbb{R}$ such that $s \nabla f(u, v) + t \nabla h(u, v) = 0$. Taking the scalar product of this expresion with $(u, 0)$, and using the fact that $f(u,v)=0=h(u,v)$, we obtain
\begin{align*}
0 &= s \langle \nabla f(u, v), (u, 0) \rangle + t \langle \nabla h(u, v) , (u, 0) \rangle \\ 
&= s \left( (2-p)\,\mu_1 \int_{\Omega} |x_2|^\gamma|u|^p + \lambda \alpha (2 - \alpha) \int_{\Omega}|x_2|^\gamma|u|^{\alpha} |v|^{\beta} \right) \\
&\quad + t \left( -\lambda \alpha \beta \int_{\Omega}|x_2|^\gamma|u|^{\alpha} |v|^{\beta} \right) \\
&=: s a_{11} + t a_{12}.
\end{align*}
Similarly, taking the scalar product with $(0, v)$ yields
\begin{align*}
0 &= s \langle \nabla f(u, v), (0, v) \rangle + t \langle \nabla h(u, v) , (0, v) \rangle \\
&= s \left( - \lambda \alpha \beta \int_{\Omega}|x_2|^\gamma|u|^{\alpha} |v|^{\beta}  \right) \\
&\quad + t  \left( (2-p)\,\mu_2 \int_{\Omega}|x_2|^\gamma |v|^p + \lambda \beta (2 - \beta) \int_{\Omega} |x_2|^\gamma |u|^{\alpha} |v|^{\beta} \right) \\
&=: s a_{21} + t a_{22}.
	\end{align*}
We consider two cases. First, if $\int_{\Omega}|x_2|^\gamma|u|^{\alpha} |v|^{\beta} = 0$, statement $(a)$ yields
\begin{equation*}
\det(a_{ij}) = (2-p)^2 \left(\mu_1|u|^{p}_{p,\gamma} \right) \left(\mu_2 |v|^{p}_{p,\gamma} \right) \geq (2-p)^2 c_0^2 >0.
\end{equation*}
On the other hand, if $\int_{\Omega}|x_2|^\gamma|u|^\alpha |v|^{\beta} \neq 0$, then, as $\lambda<0$ and $\alpha,\beta<p$,
\begin{align}
A &:= \frac{\mu_1 \int_{\Omega}|x_2|^\gamma |u|^p}{-\lambda \int_{\Omega}|x_2|^\gamma|u|^\alpha |v|^{\beta}} \geq \left(\frac{c_0}{- \lambda p \int_{\Omega}|x_2|^\gamma|u|^{\alpha} |v|^{\beta} } + 1\right)\alpha, \label{eq:A}\\
B &:= \frac{\mu_2 \int_{\Omega} |x_2|^\gamma|v|^p}{-\lambda \int_{\Omega}|x_2|^\gamma|u|^\alpha |v|^{\beta}} \geq \left(\frac{c_0}{- \lambda p \int_{\Omega}|x_2|^\gamma|u|^{\alpha} |v|^{\beta} } + 1\right)\beta. \label{eq:B}
\end{align}
Note that
$$\det(a_{ij}) = \left(\lambda \int_{\Omega}|x_2|^\gamma|u|^{\alpha} |v|^{\beta}\right)^2 \; 
	\begin{vmatrix}
	(2-p)A - \alpha(2- \alpha) & \alpha \beta \\
	\alpha \beta & (2-p)B - \beta(2- \beta)
	\end{vmatrix}$$
and
\begin{align*}
& 	\begin{vmatrix}
	(2-p)A - \alpha(2- \alpha) & \alpha \beta \\
	\alpha \beta & (2-p)B - \beta(2- \beta)
	\end{vmatrix} \\
&\qquad= (2-p)\,((2-p)AB - [\beta(2-\beta)A + \alpha(2-\alpha)B] + 2\alpha \beta) \\
&\qquad=(p-2)\alpha\beta\left((p-2)\frac{A}{\alpha}\frac{B}{\beta}-(\beta-2)\frac{A}{\alpha}-(\alpha-2)\frac{B}{\beta}-2\right).
	\end{align*}
Next, we show that there is a constant $c_1>0$ such that
\begin{equation} \label{eq:D}
D:=(p-2)\frac{A}{\alpha}\frac{B}{\beta}-(\beta-2)\frac{A}{\alpha}-(\alpha-2)\frac{B}{\beta}-2 > \frac{c_1}{-\lambda \int_{\Omega}|x_2|^\gamma|u|^{\alpha} |v|^{\beta}}.
\end{equation}
To prove this inequality, we consider two cases. Recall that $\alpha<A$, $\beta<B$ and $\alpha + \beta = p > 2$. 
\begin{itemize}
\item[(i)]If $\alpha,\beta<2$, then $D>(p-2)\frac{A}{\alpha}-(\beta-2)-(\alpha-2)-2=(p-2)\left(\frac{A}{\alpha}-1\right).$
\item[(ii)]If $\alpha\geq 2$, then 
$$D>(p-2)\frac{A}{\alpha}\frac{B}{\beta}-\beta\frac{A}{\alpha}\frac{B}{\beta}-(\alpha-2)\frac{A}{\alpha}\frac{B}{\beta}+2\left(\frac{A}{\alpha}-1\right)=2\left(\frac{A}{\alpha}-1\right).$$
Similarly, $D>2\left(\frac{B}{\beta}-1\right)$ if $\beta\geq 2$. 
\end{itemize}
Combining these inequalities with \eqref{eq:A} and \eqref{eq:B} we obtain \eqref{eq:D}. Consequently,
\begin{equation} \label{eq:det}
\det(a_{ij})>c_2\left(\int_{\Omega}|x_2|^\gamma|u|^{\alpha} |v|^{\beta}\right),
\end{equation}
for some constant $c_2>0$. So, in any case, $\det(a_{ij})>0$ and, therefore, $\nabla f(u,v)$ and $\nabla h(u,v)$ are linearly independent, as claimed.

$(c):$ If $(u,v) \in \mathcal{N}^G$ is a critical point of the restriction $E|_{\mathcal{N}^G}$, then
$$\nabla E (u, v) = s \nabla f(u, v) + t \nabla h(u, v) \quad \text{for some } s,t \in \mathbb{R}.$$
Taking the scalar product with $(u,0)$ and $(0,v)$ we obtain
	\begin{align*}
	s \langle \nabla f(u, v), (u, 0) \rangle + t \langle \nabla h(u, v), (u, 0) \rangle  &= \langle \nabla E(u, v), (u, 0) \rangle = f(u, v) = 0, \\
	s \langle \nabla f(u, v), (0, v) \rangle + t \langle \nabla h(u, v), (0, v) \rangle &= \langle \nabla E(u, v), (0, v) \rangle = h(u, v) = 0.
	\end{align*}
But, as was seen in statement $(b)$, this happens only if $s=t=0$. Hence, $\nabla E (u, v) = 0$, as claimed.
\end{proof}

We end this section with the following nonexistence result.

\begin{proposition}\label{prop:pohozhaev}
If $\Omega$ is starshaped with respect to the origin and $p\geq p_\gamma:=\frac{2(N+\gamma)}{N-2}$, the system \eqref{eq:system} does not have a nontrivial solution.
\end{proposition}

\begin{proof}
If $(u,v)$ is a solution to \eqref{eq:system}, then the identity in  \cite[Proposition 3]{pus} with $\mathscr{F}(x,u,v,y,z):=\frac{1}{2}(|y|^2+|z|^2)-\frac{1}{p}|x_2|^\gamma(\mu_1 |u|^p + \mu_2 |v|^p) - \lambda|x_2|^\gamma|u|^\alpha|v|^\beta$, $h(x):=x$ and $a:=0$, $x\in\Omega$, $u,v\in\mathbb{R}$, $y,z\in\mathbb{R}^N$, holds true. Integrating this identity over $\Omega$ and noting that $f(u,v)=0=h(u,v)$ we get
\begin{align*}
&-\frac{1}{2}\int_{\partial\Omega}\left(\left|\frac{\partial u}{\partial \nu}\right|^2+\left|\frac{\partial v}{\partial \nu}\right|^2\right)(x\cdot \nu) \\
&\qquad\quad=\frac{N-2}{2}\int_\Omega(|\nabla u|^2+|\nabla v|^2)-\frac{N+\gamma}{p}\int_\Omega|x_2|^\gamma(\mu_1 |u|^p + \mu_2 |v|^p) \\
&\qquad\qquad - \lambda(N+\gamma)\int_\Omega|x_2|^\gamma|u|^\alpha|v|^\beta \\
&\qquad\quad=\left(\frac{N-2}{2}-\frac{N+\gamma}{p}\right)\int_\Omega(|\nabla u|^2+|\nabla v|^2).
\end{align*}
Observe that $\frac{N-2}{2}\geq\frac{N+\gamma}{p}$ iff $p\geq p_\gamma$. So the result follows immediately from this identity if $p>p_\gamma$. If $p=p_\gamma$, we apply the unique continuation principle.
\end{proof}

\section{The competitive system} \label{sec:competitive}

Throughout this section we assume that $\lambda<0$. We continue to assume that $p \in (2,2^*_{N-d})$ and that $\gamma > \max\{\gamma_p,0\}$ if $\Omega_0\neq\emptyset$.

Our first goal is to prove Theorem \ref{thm:mainthm} for $\lambda<0$. Using a critical point result due to Szulkin \cite{sz} we will show that the functional $E$ restricted to $\mathcal{N}^G$ has infinitely many critical points. By Propositions \ref{prop:solutions} and \ref{prop:nehari}, they are fully nontrivial solutions to the system \eqref{eq:system}.

We write $\nabla_{\mathcal{N}^G} E (u,v)$ for the orthogonal projection of $\nabla E(u,v)$ onto the tangent space to $\mathcal{N}^G$ at $(u, v)$. 

The proofs of the following two lemmas are similar to those of the analogous statements in \cite{cp}. We include them here for the sake of completeness.

\begin{lemma} \label{lem:ps}
The functional $E:\mathcal{N}^G\to\mathbb{R}$ satisfies the Palais-Smale condition $(PS)_c$ for every $c\in \mathbb{R}$, i.e., every sequence $((u_k,v_k))$ in $\mathcal{N}^G$ such that
$$E(u_k,v_k)\to c \quad \text{and} \quad \nabla_{\mathcal{N}^G} E (u,v) \to 0,$$
contains a subsequence which converges strongly in $\mathscr{D}$.
\end{lemma}

\begin{proof}
Let $((u_k,v_k))$ be as above. It is easy to see that the sequences $((u_k, v_k))$, $(\nabla f(u_k, v_k))$ and $(\nabla h(u_k, v_k))$ are bounded in $\mathscr{D}^G$. 

We claim that $\nabla E(u_k,v_k)\to 0$. To prove this claim, we write
	\begin{equation}\label{eq:grad}
	\nabla E(u_k, v_k) = \nabla_{\mathcal{N}^{G}} E(u_k, v_k) + s_k \nabla f(u_k) + t_k \nabla h(u_k),
	\end{equation}
with $s_k, t_k \in \mathbb{R}$. Taking the scalar product of this identity with $(u_k,0)$ and $(0,v_k)$, we see that $s_k$ and $t_k$ solve the system
	\begin{equation}\label{eq:system_grad}
	\begin{cases}
	o(1) = s_k a_{11}^{(k)} + t_k a_{12}^{(k)}, \\
	o(1) = s_k a_{21}^{(k)} + t_k a_{22}^{(k)},
	\end{cases}
	\end{equation}
where $o(1)\to 0$ as $k\to\infty$,
	\begin{align*}
	a_{11}^{(k)} &:= (2-p)\mu_1 \int_{\Omega} |x_2|^{\gamma}|u_k|^p + \lambda \alpha(2- \alpha) \int_{\Omega} |x_2|^{\gamma} |u_k|^{\alpha}|v_k|^{\beta}, \\
	a_{12}^{(k)} &:= - \lambda \alpha \beta \int_{\Omega} |x_2|^{\gamma} |u_k|^{\alpha}|v_k|^{\beta} =: a_{21}^{(k)}, \\
	a_{22}^{(k)} &:= (2-p)\mu_2 \int_{\Omega} |x_2|^{\gamma}|v_k|^p + \lambda \beta(2- \beta) \int_{\Omega} |x_2|^{\gamma} |u_k|^{\alpha}|v_k|^{\beta}.
	\end{align*}
After passing to a subsequence, we have that $\int_{\Omega} |x_2|^{\gamma} |u_k|^{\alpha}|v_k|^{\beta}\to b\in[0,\infty)$. If $b=0$, Proposition \ref{prop:nehari}$(a)$ implies that
$$\det(a_{ij}^{(k)})\geq \frac{1}{2}(2-p)^2 c_0^2 \qquad\text{for }k\text{ large enough}.$$
If $b\neq 0$, statement \eqref{eq:det} yields 
$$\det(a_{ij}^{(k)})\geq \frac{c_2}{2}b>0 \qquad\text{for }k\text{ large enough}.$$
Hence, after passing to a subsequence, we have that $s_k\to 0$ and $t_k\to 0$, and from \eqref{eq:grad} we get that $\nabla E(u_k,v_k)\to 0$, as claimed.

By Theorem \ref{thm:in}, the embedding $\mathscr{D}^G \hookrightarrow L^p(\Omega;|x_2|^\gamma)\times L^p(\Omega;|x_2|^\gamma)$ is compact. Using this fact, it is now standard to show that $((u_k,v_k))$ contains a subsequence which converges strongly in $\mathscr{D}^G$.
\end{proof}

Let $Z$ be a symmetric subset of $\mathcal{N}^G$, i.e.,  $(-u,-v)\in \mathcal{N}^G$ iff $(u,v)\in \mathcal{N}^G$. If $Z\neq\emptyset$, the genus of $Z$ is the smallest integer $j \geq 1$ such that there exists an odd continuous function $Z \to \mathbb{S}^{j-1}$ into the unit sphere $\mathbb{S}^{j-1}$ in $\mathbb{R}^{j}$. We denote it by $\mathrm{genus}(Z)$. If no such $j$ exists, we define $\mathrm{genus}(Z) := \infty$. We set $\mathrm{genus}(\emptyset):=0$.

\begin{lemma} \label{lem:genus}
For all $j\geq 1$,
$$\Sigma_j := \{Z\subset \mathcal{N}^G : Z \text{ is symmetric and compact, and }\,\mathrm{genus}(Z)\geq j\}\neq\emptyset.$$
\end{lemma}

\begin{proof}
If $u,v \in D_0^{1,2}(\Omega)^G$, $u,v \neq 0$, we denote by $s_u, t_v$ the unique positive numbers such that $\|s_u u\|^2=\mu_1|s_u u|_{p,\gamma}^p$ and $\|t_v v\|^2=\mu_2|t_v v|_{p,\gamma}^p$. Note that $(s_u u, t_v v) \in \mathcal{N}^G$ if $uv = 0$. We write
$$\rho(u,v) := (s_u u, t_v v).$$

Given $j\geq 1$, we choose $G$-invariant functions $u_1,\ldots,u_j,v_1,\ldots,v_j$ in $\mathcal{C}_c^{\infty}(\Omega)$ such that any two of them have disjoint supports. Let $\{ e_i : 1 \leq i \leq j \}$ be the canonical basis of $\mathbb{R}^j$, and $Q$ be the convex hull of $\{\pm e_i : 1 \leq i \leq j \}$, i.e.,
$$Q := \left\{\sum_{i=1}^{j} r_i \hat{e}_i : \hat{e}_i \in \{\pm e_i\}, \, r_i \in [0,1], \, \sum_{i=1}^{j} r_i = 1 \right\},$$
which is homeomorphic to the unit sphere $\mathbb{S}^{j-1}$ in $\mathbb{R}^{j}$ by an odd homeomorphism.

We define $\sigma: Q \to\mathcal{N}^{\Gamma}$ by setting $\sigma(e_i) := (u_i,v_i)$, $\sigma(-e_i):= (-u_i,-v_i)$, and
$$\sigma\left(\sum_{i=1}^{j} r_i \hat{e}_i \right) := \rho \left(\sum_{i=1}^{j} r_i \sigma(\hat{e}_i) \right).$$
This is a well defined, continuous, odd map. Hence, $Z := \sigma(Q)$ is a symmetric compact subset of $\mathcal{N}^G$. If $\tau: Z \to\mathbb{S}^{k-1}$ is an odd continuous map, the composition $\tau \circ \sigma$ yields an odd continuous map $\mathbb{S}^{j-1} \to \mathbb{S}^{k-1}$, which, by the Borsuk-Ulam theorem, forces $k \geq j$. This shows that $\mathrm{genus}(Z) \geq j$. Thus, $Z\in\Sigma_j$. 
\end{proof}

\begin{proof}[Proof of Theorem \ref{thm:mainthm} for $\lambda<0$]
By Proposition \ref{prop:nehari} and Lemma \ref{lem:ps}, $\mathcal{N}^G$ is a closed symmetric $\mathcal{C}^1$-submanifold of $\mathscr{D}^G$ that does not contain the origin, and $E$ is an even $\mathcal{C}^1$-functional, which is bounded below on $\mathcal{N}^G$ by a positive constant and satisfies $(PS)_c$ for every $c\in \mathbb{R}$. Since, by Lemma \ref{lem:genus}, $\Sigma_j\neq\emptyset$ for every $j \geq 1$, Szulkin's multiplicity result \cite[Corollary 4.1]{sz} implies that $E$ attains its minimum and has infinitely many critical points on $\mathcal{N}^G$. As $E(u,v) = E(|u|,|v|)$, the minimum can be chosen to be positive.
\end{proof}

To prove Theorem \ref{thm:separation} we could follow the argument of \cite[Proposition 5.1]{cp}. A simpler argument is given next. To highlight the role played by $\lambda$ we write $E_{\lambda}, f_{\lambda}, h_{\lambda}, \mathcal{N}^G_{\lambda}$ instead of $E, f, h,  \mathcal{N}^G$, and set
$$c_{\lambda}^G:= \inf_{(u,v) \in \mathcal{N}^G_\lambda} E_{\lambda}(u,v).$$

The nontrivial $G$-invariant solutions to the problem \eqref{eq:signchanging} are the critical points of the restriction of the functional $J : D_0^{1,2}(\Omega)^G \to \mathbb{R}$ to the Nehari manifold $\mathcal{M}^G$, defined as
\begin{align} 
J(w) &:= \frac{1}{2} \int_{\Omega} |\nabla w|^2 - \frac{1}{p} \int_{\Omega} |x_2|^\gamma(\mu_1 |w^+|^p + \mu_2 |w^-|^p), \label{eq:J} \\
\mathcal{M}^G &:= \{ w \in D_0^{1,2}(\Omega)^G: w \neq 0, \int_{\Omega} |\nabla w|^2 = \int_{\Omega} |x_2|^\gamma(\mu_1 |w^+|^p + \mu_2 |w^-|^p) \}. \nonumber
\end{align} 
The sign-changing solutions lie on the set
$$\mathcal{E}^G:= \{ w \in D_0^{1,2}(\Omega)^G: w^+, w^- \in \mathcal{M}^G\}.$$
We define
$$c^G_{\infty}:= \inf_{w \in \mathcal{E}^G} J(w).$$
Note that, if $w \in \mathcal{E}^G$, then, as $w^+ w^- = 0$, we have that $J(w) = E(w^+, w^-)$ and $(w^+,w^-) \in \mathcal{N}^G_{\lambda}$. Therefore, $c_{\lambda}^G \leq c_{\infty}^G$ for every $\lambda<0$.

\begin{proof}[Proof of Theorem \ref{thm:separation}]
Let $\lambda_k\to -\infty$ and $(u_k,v_k)\in\mathcal{N}^G_{\lambda_k}$ satisfy $E(u_k,v_k)=c_{\lambda_k}^G$ and $u_k,v_k\geq 0$. Then,
$$\frac{p-2}{2p} \int_{\Omega} \left(|\nabla u_k|^2 + |\nabla v_k|^2\right) = c_{\lambda_k}^G \leq c_{\infty}^G.$$
So, after passing to a subsequence,
\begin{align*}
u_k \rightharpoonup u_{\infty}, \qquad v_k \rightharpoonup v_{\infty}, & \qquad\text{weakly in } D_{0}^{1,2}(\Omega)^G, \\
u_k \to u_{\infty}, \qquad v_k \to v_{\infty},  &\qquad\text{strongly in } L^{p}(\Omega;|x_2|^\gamma), \\
u_k \to u_{\infty}, \qquad v_k \to v_{\infty}, & \qquad\text{a.e. in } \Omega.
\end{align*}
Hence, $u_{\infty}, v_{\infty} \geq 0$. Moreover, as $f_{\lambda_k}(u_k,v_k)+h_{\lambda_k}(u_k,v_k)=0$, we have that
$$0 \leq \int_{\Omega}|x_2|^\gamma|u_k|^{\alpha}|v_k|^{\beta} \leq \frac{1}{p(-\lambda_k)} \int_{\Omega}|x_2|^\gamma \left(\mu_1|u_k|^p + \mu_2|v_k|^p \right) \leq \frac{C_0}{(-\lambda_k)},$$
and from Fatou's lemma we obtain 
$$0 \leq \int_{\Omega} |x_2|^\gamma|u_{\infty}|^{\alpha}|v_{\infty}|^{\beta} \leq \lim_{k \to \infty} \int_{\Omega}|x_2|^\gamma|u_k|^{\alpha} |v_k|^{\beta} = 0.$$
Hence, $u_{\infty} v_{\infty} = 0$ a.e. in $\Omega$. On the other hand, by Proposition \ref{prop:nehari},
$$0<c_0 \leq \|u_k\|^2 \leq \mu_1|u_k|_{p,\gamma}^p \quad \text{and} \quad 0<c_0 \leq \|v_k\|^2\leq \mu_2|v_k|_{p,\gamma}^p.$$
So, passing to the limit, we obtain that $u_\infty\neq 0$ and $v_\infty\neq 0$. Moreover,
\begin{equation} \label{eq:comparison}
\|u_\infty\|^2 \leq \mu_1|u_\infty|_{p,\gamma}^p \quad \text{and} \quad \|v_\infty\|^2\leq \mu_2|v_\infty|_{p,\gamma}^p.
\end{equation}
Let $s,t$ be the unique positive numbers such that $\|su_\infty\|^2 = \mu_1|su_\infty|_{p,\gamma}^p$ and $\|tv_\infty\|^2 = \mu_2|tv_\infty|_{p,\gamma}^p$. Then, $su_\infty - tv_\infty \in \mathcal{E}^G$. The inequalities \eqref{eq:comparison} imply that $s,t\in (0,1]$. Therefore,
\begin{align*}
c_\infty^G &\leq \frac{p-2}{2p}\left(\|su_\infty\|^2+\|tv_\infty\|^2\right) \leq \frac{p-2}{2p}\left(\|u_\infty\|^2+\|v_\infty\|^2\right) \\
&\leq \frac{p-2}{2p}\lim_{k\to\infty}\left(\|u_k\|^2+\|v_k\|^2\right)=\lim_{k\to\infty} c_{\lambda_k}^G \leq c_{\infty}^G.
\end{align*}
It follows that $u_k \to u_{\infty}$ and $v_k \to v_{\infty}$ strongly in $D_{0}^{1,2}(\Omega)^G$, $s=t=1$, $u_\infty - v_\infty \in \mathcal{E}^G$ and $J(u_\infty - v_\infty)=c_{\infty}^G$. Arguing as in \cite[Lemma 2.6]{ccn} we conclude that $u_\infty - v_\infty$ is a least energy $G$-invariant sign-changing solution to the problem \eqref{eq:signchanging}.
\end{proof}

\section{The equation and the cooperative system} \label{sec:cooperative}

First, we prove Theorem \ref{thm:equation}.

\begin{proof}[Proof of Theorem \ref{thm:equation}]
The nontrivial $G$-invariant solutions to the problem \eqref{eq:Henon} are the critical points of the functional $J$ restricted to the Nehari manifold $\mathcal{M}^G$, as defined in \eqref{eq:J} with $\mu_1=\mu_2=1$. $J$ is an even $\mathcal{C}^2$-functional, which is bounded below on $\mathcal{M}^G$ by a positive constant. Since $p \in (2,2^*_{N-d})$ and $\gamma > \max\{\gamma_p,0\}$ if $\Omega_0\neq\emptyset$, Theorem \ref{thm:in} implies that $J$ it satisfies $(PS)_c$ for every $c\in \mathbb{R}$. Standard variational methods yield a positive minimizer for $J$ on $\mathcal{M}^G$, and one can easily adapt \cite[Theorem 3.7]{cpa0} to show that $J$ has infinitely many sign-changing $G$-invariant critical points, as claimed.
\end{proof}

The next lemma yields Theorem \ref{thm:mainthm} for cooperative systems.

\begin{lemma}
Let $w$ be a nontrivial solution to the problem \eqref{eq:Henon}. Then, there exist $s,t>0$ such that $(sw,tw)$ is a solution to the system \eqref{eq:system} if and only if there exists $r>0$ such that
\begin{equation} \label{eq:coop}
\mu_1 r^{p-2} + \lambda\alpha r^{\alpha -2}-\lambda\beta r^\alpha -\mu_2=0 \qquad\text{and}\qquad \mu_2+\lambda\beta r^\alpha>0.
\end{equation}
\end{lemma}

\begin{proof}
The proof is completely analogous to that of \cite[Lemma 4.1]{cf}.
\end{proof}

\begin{proof}[Proof of Theorem \ref{thm:mainthm} for $\lambda>0$]
The inequality in \eqref{eq:coop} is automatically satisfied if $\lambda>0$, so \eqref{eq:coop} reduces to assumption $(A)$ and Theorem \ref{thm:mainthm} follows from Theorem \ref{thm:equation}.
\end{proof}

\medskip

\begin{flushleft}
\textbf{Omar Cabrera} / \textbf{Mónica Clapp}\\
Instituto de Matemáticas\\
Universidad Nacional Autónoma de México\\
Circuito Exterior, Ciudad Universitaria\\
04510 Coyoacán, CDMX\\
Mexico\\
\texttt{omar.cabrera@im.unam.mx} / \texttt{monica.clapp@im.unam.mx} 
\end{flushleft}

\end{document}